\pdfoutput=1

\documentclass[12]{amsart}

\usepackage{amsmath}
\usepackage{amsfonts}
\usepackage{rotating}
\usepackage{bbm}
\usepackage[all,cmtip]{xy}
\usepackage{subcaption}
\usepackage{amssymb}
\usepackage[latin1]{inputenc}
\usepackage{graphicx}
\usepackage{dsfont}
\usepackage{color}
\usepackage{hyperref}

\renewcommand{\phi}{\varphi}
\newcommand{\C}{{\mathbb{C}}}
\newcommand{\R}{{\mathbb{R}}}

\advance\hoffset by -0.5 truecm
\usepackage{amssymb}
\newtheorem{Theorem}{Theorem}[section]
\newtheorem{Lemma}[Theorem]{Lemma}
\newtheorem{Corollary}[Theorem]{Corollary}

\newtheorem{Proposition}[Theorem]{Proposition}

\newtheorem{Remark}[Theorem]{Remark}

\def\R{{\mathbb R}}

\def \C{{\mathbb C}}

\def \0{\lambda_{0}}

\begin{document}

\title[The Euler problem of two fixed centers]{   On a convex embedding of the Euler problem of two fixed centers  }

\author{Seongchan Kim}
 \address{Universit\"at Augsburg, Universit\"atsstrasse 14, D-86159 Augsburg, Germany}
 \email {seongchan.kim@math.uni-augsburg.de}
\date{\today}



\begin{abstract}
In this article, we study a convex embedding for   the Euler problem of two fixed centers for energies below the critical energy level. We prove that  the doubly-covered elliptic coordinates provide a 2-to-1 symplectic embedding such that the image of the bounded component near the lighter primary of the regularized Euler problem is convex for any energy below the critical Jacobi  energy. This holds true if the two primaries have the equal mass, but  does not holds near the heavier body. 
\end{abstract}

\maketitle
\setcounter{tocdepth}{1}
\tableofcontents

\section{Introduction}\label{sec1}
The \textit{Euler problem of two fixed centers} describes the behavior of a massless body which is attracted by two fixed primaries. The massless body will be referred to as the \textit{satellite} and the two primaries as the \textit{Earth} and the \textit{Moon}. Its Hamiltonian $H : ( \R^2 \setminus \left \{E, M \right \} ) \times \R^2 \rightarrow \R$ is given by
\begin{equation}\label{eq:Hamiltonian}
H(q,p) = \frac{1}{2}|p|^2- \frac{ 1-\mu}{|q-E|} - \frac{\mu}{|q-M|},
\end{equation}
where $\mu \in (0,1)$ is the mass ratio of the two primaries: $E=(0,0)$ and $M=(1,0)$. We observe that the Hamiltonian is \textit{mechanical}, i.e., the sum of the kinetic energy $(1/2)|p|^2$ and the potential energy $U(q):=-(1-\mu)/|q-E| - \mu/|q-M|$. If $\mu<1/2$, then the Earth is heavier than the Moon and if $\mu>1/2$, then the Moon is heavier.  For a negative energy, the satellite is confined to a bounded region in the configuration space. 

There exists a precisely one critical point $L=(l, 0,0,0)$, where 
\begin{equation*}
l = \begin{cases} \frac{ 1 - \mu - \sqrt{ \mu(1-\mu)}}{1-2\mu} & \text{ if } \mu \neq 1/2 \\ 1/2 & \text{ if } \mu = 1/2.\end{cases}
\end{equation*}
The energy value $c_J := H(L) = -1 - 2\sqrt{\mu(1-\mu)}$ is referred to as the \textit{critical Jacobi energy}. In this article, we only consider  energies less than this energy level, $c<c_J$. 

For $c<c_J$, the energy hypersurface $H^{-1}(c)$ consists of two bounded components, whose closures are neighborhoods of the Earth and the Moon. We abbreviate by $\Sigma_c^E$ resp. $\Sigma_c^M$ the Earth component resp. the Moon componenty. Let $\pi : T^* \R^2 \rightarrow \R^2 $ be the projection along the fiber. Given an energy $c$, the \textit{Hill's region} is defined by
\begin{equation*} 
\mathcal{K}_c := \pi( H^{-1}(c) ) \subset \R^2 \setminus \left \{ E, M \right \},
\end{equation*}
which  also consists of  two bounded components: the Earth component $\mathcal{K}_c^E$ and the Moon component $\mathcal{K}^M_c$. By means  of the equation (\ref{eq:Hamiltonian}), one can  write the Hill's region as
\begin{equation*}
\mathcal{K}_c = \left \{ q=(q_1, q_2) \in \R^2 \setminus \left \{ E, M \right \} : U(q) \leq c\right \}.
\end{equation*}

Since the Earth and the Moon are fixed, one can regard them as the foci of a set of ellipses and hyperbolas. To introduce the elliptic coordinates it is convenient to apply the translation $(q_1, q_2, p_1, p_2) \mapsto (q_1-1/2, q_2, p_1, p_2)$ so that   $E=(-1/2, 0)$ and $M=(1/2,0)$. The doubly-covered elliptic coordinates $(\lambda, \nu)  \in \R \times S^1[-\pi, \pi]$ are now defined by
\begin{equation}
\cosh \lambda = |q-E| + |q-M| \in [1, \infty) \;\;\; \text{and}\;\;\; \cos\nu = |q-E| - |q-M| \in [-1, 1],
\end{equation}
whose inverse  is given by
\begin{equation}
q_1 = \frac{1}{2}\cosh \lambda \cos \nu \;\;\; \text{and}\;\;\; q_2 = \frac{1}{2} \sinh \lambda \sin \nu.
\end{equation}
The momenta $p_{\lambda}$ and $p_{\nu}$ are determined by the canonical relation $p_{\lambda} d\lambda + p_{\nu} d \nu = p_1 dq_1 + p_2 dq_2$. Then the Hamiltonian becomes 
\begin{equation} \label{eq:elliptic}
H = \frac{H^1 + H^2}{\cosh^2 \lambda - \cos^2 \nu},
\end{equation}
where $H^1 = 2p_{\lambda}^2 - 2\cosh\lambda $ and $H^2 = 2p_{\nu}^2 + 2(1-2\mu)\cos \nu.$  We define the regularized Hamiltonian by 
\begin{equation}\label{regularizedelliptic}
Q_c = (H-c)(\cosh^2 \lambda - \cos^2 \nu) = Q^1 + Q^2,
\end{equation}
where $Q^1 = 2p_{\lambda}^2 -2\cosh\lambda -c \cosh^2\lambda$ and $Q^2= 2p_{\nu}^2 + 2(1-2\mu)\cos \nu + c \cos^2 \nu$.  Our main result is the following  

\begin{Theorem}\label{theorem3} The doubly-covered elliptic coordinates provide a 2-to-1 symplectic embedding satisfying that the image of the regularized Earth component is convex for any energy below the critical Jacobi energy, provided that $\mu \geq 1/2$. If $\mu < 1/2$, then there exists an energy level $c_0 = c_0(\mu)<c_J$ such that for $c<c_0$ the regularized Earth component is convex and for $c\geq c_0$ it is not. 
\end{Theorem}

Recall that in \cite{HWZ} Hofer, Wysocki and Zehnder proved that  if  $M \subset \R^4$ is a closed strictly convex hypersurface, then $M$ admits a disk-like global surface of section, see Section \ref{sec2}. In view of this result, the previous theorem implies that for every energy below the critical Jacobi energy the image of the regularized bounded component
around the lighter primary under the doubly-covered elliptic coordinates admits
a disk-like global surface of section.  In the upcoming paper \cite{Kim2} we will see that this is also the case around the heavier primary.

\begin{Remark}  \rm  Let us switch on the rotation term, more precisely for $ a \in [0,1]$, consider the family of  Hamiltonians $H_a(q,p) = H(q,p) + a (q_1 p_2 - q_2p_1)$, where $H$ is given as \eqref{eq:Hamiltonian}. If $a=0$, we obtain the Euler problem. For $a=1$, the corresponding Hamiltonian system is called the \textit{planar circular restricted three-body problem} (PCR3BP).      The Euler problem was first introduced by Euler \cite{Euler1, Euler2} as a starting point toward the study of the PCR3BP. In this paper, following his direction we study the Euler problem. Note that the assertion of the previous theorem also holds true for the Hamiltonian system associated with $H_a$ for $a>0$ sufficiently small since we just add a regular function.  The ultimate goal is to prove the existence of a convex embedding for the PCR3BP to prove the Birkhoff conjecture: if  energy is less than the first critical level, then the energy hypersurface of the PCR3BP contains   two bounded components. In \cite{B1},   Birkhoff proved the existence of the {retrograde periodic orbit} on each bounded component. He then conjectured that the retrograde periodic orbit bounds a {disk-like global surface  of section}. The Birkhoff conjecture holds true the Euler problem, see \cite{Kim2}. In this case, the retrograde periodic orbit becomes an obvious collision orbit which lies on the $q_1$-axis on the opposite side of the other primary. 
\end{Remark}

 The standard way to construct  the double covering  is via the  Levi-Civita embedding \cite{Levi}. The Levi-Civita coordinates are given by $q = 2v^2$ and $p = u/\overline{v}$. We define the regularized Hamiltonian $K_c : \R^4 \cong \C^2 \rightarrow \R$ by
\begin{equation}\label{eq:Levi}
K_c(v, u)  := |v|^2 ( H(v, u) -c)  = \frac{1}{2}|u|^2 - c|v|^2 - \frac{\mu |v|^2}{ \sqrt{ 4v_1^4 + 8v_1^2v_2^2 - 4v_1^2 + 4v_2^4 + 4v_2^2 +1   }} - \frac{1-\mu}{2}.
\end{equation}
Note that the bounded  components $\Sigma_c^E, \;\Sigma_c^M \subset H^{-1}(c)$ are compactified to $\overline{\Sigma}_c^E, \; \overline{\Sigma}_c^M  \subset K^{-1}_c(0)$, which are diffeomorphic to $S^3$. One can naturally ask if this standard way also gives rise to a convex symplectic embedding. The next theorem gives a partial answer.

\begin{Theorem}\label{theorem2} The image of the Levi-Civita embedding of the regularized Earth component is not convex for an energy close to the critical Jacobi energy and for the mass ratio  less than $16/17$.
\end{Theorem}

\;

As observed in \cite{contact}, for energies below the first critical energy the bounded components of the regularized energy hypersurface of the PCR3BP are star-shaped. One can ask if the starshapedness can be improved to the stronger condition, namely fiberwise convexity. To explain what the fiberwise convexity of the energy hypersurface means, without loss of generality we focus on $\Sigma_c^E$. Due to the singularity at the Earth, $\Sigma_c^E$ is noncompact. However, one can regularize it in the sense of Moser, see \cite{Moser}: switching the roles of $q$ and $p$ we regard $p$ as the position and $q$ as the momentum. The Earth component is then a subset $\Sigma_c^E \subset T^* \R^2 \subset T^*S^2$. In the last inclusion, via the stereographic projection we think $S^2 = \R^2 \cup \left\{ \infty \right\}$. The regularized Earth component, i.e., the closure $\overline{\Sigma_c^E}$ in $T^*S^2$, is regular in the fiber over $\infty$. In particular, $\overline{\Sigma_c^E} \subset T^*S^2$ is a subbundle over $S^2$ whose fibers are diffeomorphic to a circle. Moreover, the fiber over $\infty$ is a circle and hence it is convex. Note that the points in the fiber over $\infty \in S^2 $ correspond to collision state  $q=E$ at which the momenta explode. 

 The submanifold  $\overline{\Sigma_c^E} \subset T^*S^2$  is said to be \textit{fiberwise convex} if each fiber bounds a  strictly convex domain in the tangent space. Since convexity is preserved under chart transition, which is a linear transformation, to show that $\overline{\Sigma_c^E}$ is fiberwise convex it suffices to show that $\left \{ q \in \R^2 : H(q,p) = c\right\}$ bounds a strictly convex domain in $\R^2$ for each $p$ (cf. \cite{book}). 

The question if the regularized PCR3BP is fiberwise convex for energies below the first critical energy is not answered yet. If this holds true, one can interpret PCR3BP for energies below the first critical energy as the Legendre transformation of a geodesic problem on the two-sphere with a family of Finsler metrics. In particular, the integral curves of PCR3BP  can be regarded as geodesics for the Finsler metric. Moreover, the Conley-Zehnder indices of the periodic orbits agree with the Morse indices of the corresponding geodesics. A direct consequence of this is that for energies below the first critical energy, all periodic orbits  are of Conley-Zehnder indices greater than or equal to zero. Our final result partially answers this question for the Euler problem.

\begin{Theorem}\label{theorem1}  Assume that $\mu <1/2$. The regularized Earth component $\overline{\Sigma_c ^E}$ is not fiberwise convex for an energy close to the critical Jacobi energy. On the other hand, if the two primaries have the equal mass, $\mu=1/2$, then the  compact components are fiberwise convex.
\end{Theorem}

\begin{Remark} \rm By numerical experiments we expect that the compact component near the lighter body is fiberwise convex. 
\end{Remark}

To prove the theorem, as mentioned before one needs to show that for each $p$, the curve $H_p^{-1}(q)=c$ bounds a strictly convex domain, where $H_p(q) := H(q,p)$. Observe that 
\begin{equation}
\left \{ q \in \R^2 \setminus \left \{E, M \right \} :H_p^{-1}(q) = c \right \}  = \left \{ q \in \R^2 \setminus \left \{E, M \right \} : U(q) = c-\frac{1}{2}|p|^2 \right \}.
\end{equation}
Since the Hamiltonian is mechanical,   fiberwise convexity of $\overline{\Sigma_c^E}$ is equivalent to convexity of $\mathcal{K}_{c-(1/2)|p|^2}^E$ for each $p$. Moreover, since $U(q) = c- (1/2)|p|^2$ is a planar curve, we can compute its curvature.  In Section \ref{sec4} we prove that for $\mu < 1/2$ the curvature of $\mathcal{K}_{c_J}^E$ is positive far away from the point $(l,0)$ and is negative near that point. Then by continuity this proves the first assertion of the theorem. For the case $\mu=1/2$, we show that the curvature is nonvanishing along $\partial \mathcal{K}_{c- (1/2)|p|^2}^E$ for any $p$.

\;\;

\textbf{Acknowledgements:} First and foremost, I would like to express my deepest gratitude to  my advisor Urs Frauenfelder for interesting me in this subject and a lot of advice. I am also grateful to Yehyun Kwon  for fruitful discussions. Furthermore, I want to thank to the Institute for Mathematics of University of Augsburg for providing a supportive research environment. This work is supported by DFG grants CI 45/8-1 and FR 2637/2-1.

\section{History and Known results}\label{sec2}

Let $M$ be a closed  three-manifold equipped with a vector field $X$ without rest points. A \textit{global disk-like surface of section} for $X$ consists of an embedded closed disk $\mathcal{D} \subset M$ satisfying the following properties:
\begin{itemize}
\item the vector field $X$ is tangent to the boundary $\partial \mathcal{D}$. The boundary is called the \textit{spanning orbit};
 
\;

\item the interior of the disk $\mathcal{D} \setminus \partial \mathcal{D}$ is transversal to the flow of $X$;

\;

\item  every orbit, other than the spanning orbit, intersects the interior of the disk in forward and backward time. 
\end{itemize}
The existence of a global surface of section allows one to study the dynamics by means of the global Poincar\'e return map, which is conjugated to an area-preserving diffeomorphism.  

The concept of a global surface of section was introduced by Poincar\'e \cite{Poin}. Poincar\'e's original global surface of section was annulus-like and Birkhoff  generalized  its notion to the one of arbitrary genus and with an arbitrary number of boundary components \cite{B2}.  In \cite{Conley} Conley proved that in the PCR3BP  there exists   a sufficiently small energy level $c<0$, which is independent of the mass ratio, such that the energy hypersurface with energy level less than $c$ admits a annulus-like global  surface of section. Shortly after, McGehee \cite{McGehee} used the Levi-Civita embedding to study the PCR3BP in the double covering and constructed disk-like global  surfaces of section around the heavy primary for energies below the first critical level, provided that the mass ratio is sufficiently small. 

We remark that Conley and McGehee used perturbative methods. Perturbative methods hold only for   either small energy values or  small mass ratios. Under this assumption, the PCR3BP can be regarded as a perturbation of the Kepler problem. Note that one cannot apply perturbative methods   for higher energy values  or higher mass ratios.

The first try to use global methods, more precisely, holomorphic curve techniques, to attack the PCR3BP was given by  Albers-Frauenfelder-van Koert-Paternain \cite{contact}. In that paper, they proved that for energies below and slightly above the first critical level, the regularized energy hypersurfaces of the PCR3BP admit  a compatible contact form. In view of this result, one can apply the holomorphic curve theory in the symplectization  of a closed contact three-manifold, which was introduced by Hofer \cite{Hofer}, to construct a disk-like global surface of section. More precisely, we use the following result

\begin{Theorem}\label{HWZconvex}{\rm{(Hofer-Wysocki-Zehnder, \cite{HWZ})}} Assume that $M \subset \R^4$ is a closed strictly convex hypersurface. Then $M$ admits a disk-like global surface of section.
\end{Theorem}

\begin{Remark}\label{rmkopen} \rm  In fact, Hofer-Wysocki-Zehnder proved more precisely that a convex hypersurface in $\R^4$ is dynamically convex and if  a closed star-shaped hypersurface in $\R^4$ is dynamically convex, then it admits a  disk-like global surface of section. Note that the notion of dynamical convexity is preserved under symplectomorphisms, but convexity is not.
\end{Remark}

Keeping  the result in \cite{contact} and Theorem \ref{HWZconvex} in mind, Albers-Fish-Frauenfelder-Hofer-van Koert proved the following.

\begin{Theorem} \label{globaltheorem} {\rm{(Albers-Fish-Frauenfelder-Hofer-van Koert, \cite{GSS})}}  For energies below the first critical value, the bounded component around the lighter primary of the Levi-Civita embedding of the regularized planar circular restricted three-body problem is convex, provided that the mass ratio is sufficiently small.
\end{Theorem}

However, convexity of the Levi-Civita embedding fails if the mass ratio equals zero.
\begin{Theorem} \label{convexnot} {\rm{{(Albers-Fish-Frauenfelder-van Koert, \cite{RKP})}}} The image of the Levi-Civita embedding of the regularized rotating Kepler problem is not convex for energies close to the critical Jacobi energy level.
\end{Theorem}

Recently,  Frauenfelder-van Koert-Zhao composed the Levi-Civita embedding with the Ligon-Schaaf mapping  and proved the following.

\begin{Theorem}\label{zhao} {\rm{(Frauenfelder-van Koert-Zhao, \cite{convexRKP})}} For energies below the critical Jacobi energy, a proper combination of the Ligon-Schaaf and Levi-Civita regularization mappings provides a convex symplectic embedding of the energy hypersurfaces of the rotating Kepler problem into $\R^4$.
\end{Theorem}

On the other hand, as mentioned in the introduction, one can ask if the regularized PCR3BP is fiberwise convex. Cieliebak-Frauenfelder-van Koert give a positive answer to this question,  provided that the mass ratio equals zero. 

\begin{Theorem}\label{finslder}{\rm{(Cieliebak-Frauenfelder-van Koert, \cite{Finsler})}} For energies below the critical Jacobi energy, the bounded components of the regularized rotating Kepler problem are fiberwise convex.
\end{Theorem}

The Hill's lunar problem was introduced by Hill \cite{Hill} to study the stability of orbits of the Moon. This problem is a limit case of PCR3BP where the mass of the heavier primary diverges to the infinity and the massless body moves in a very small neighborhood of the lighter primary. For this problem, the result on fiberwise convexity was proved by Lee.

\begin{Theorem}\label{Lee}{\rm{(Lee, \cite{Lee})}} The bounded components of the regularized Hill's lunar problem are fiberwise convex for energies below the critical value.
\end{Theorem}

\section{Proof of  Theorem \ref{theorem3}}\label{sec6}

We assume that $\mu \leq 1/2$ so that either the Earth is heavier than the Moon or the two primaries have the equal mass. Recall that the regularized Hamiltonian in the doubly-covered elliptic coordinates is given by
\begin{equation}
Q  = (H-c)(\cosh^2 \lambda - \cos^2 \nu) = Q^1 + Q^2,
\end{equation}
where $Q^1 = 2p_{\lambda}^2 -2\cosh\lambda -c \cosh^2\lambda$ and $Q^2= 2p_{\nu}^2 + 2(1-2\mu)\cos \nu + c \cos^2 \nu$.  To prove  Theorem \ref{theorem3},  we use the following criterion which shows that the condition for an energy hypersurface  to be convex can be expressed in terms of the potential function.

\begin{Theorem}\label{theoremsalomao}{\rm{(Pedro A. S. Salom$\tilde{a}$o, \cite{salomao})}} Let $H: \R^4 \rightarrow \R$ be a mechanical Hamiltonian. Assume that the potential function, denoted by $V= V(x,y)$, is a $C^{k\geq 2}$ function. Suppose that $S\subset H^{-1}(c)$ is homeomorphic to $S^3$, invariant by the Hamiltonian flow and has at most one singularity $p_c$. Let $\pi : \R^4 \rightarrow \R^2 $ be the projection along the fiber. Then $S$ is strictly convex if and only if 
\begin{equation}\label{eq:convex}
2(c-V)(V_{xx}V_{yy} - V_{xy}^2) + V_{xx}V_y^2 + V_{yy}V_x^2 - 2 V_x V_y V_{xy} >0
\end{equation}
for all points in $\pi(S) \setminus \pi(p_c)$. 
\end{Theorem}

We apply this theorem to the potential 
$$
V(\lambda, \nu) = -\frac{1}{2}\cosh \lambda - \frac{c}{4}\cosh^2 \lambda + \frac{1-2\mu}{2}\cos\nu + \frac{c}{4} \cos^2\nu
$$
of the Hamiltonian $Q/4$ and it then suffices to show that  $A(\cosh \lambda, \cos\nu)$ is nonvanishing, where the function $A$ is defined by
\begin{eqnarray*}
A&:=&(cx^2 + 2x - cy^2 - 2(1-2\mu)y)( 2c x^2 +x -c)( 2cy^2 + (1-2\mu)y -c) \\
&&- (1-y^2)( 1-2\mu +cy)^2 ( 2cx^2 +x -c)- (x^2-1)(1+cx)^2 ( 2cy^2 + (1-2\mu)y -c).
\end{eqnarray*}
Recall that the domain of the function $A$ is given as follows  
\begin{equation}\label{ellipticdomain}
\begin{cases}    x \in [1, \frac{ -1-\sqrt{  c^2 - 2(1-2\mu)c +1}}{c}], \; y \in [-1, \frac{ -(1-2\mu) + \sqrt{ c^2 +2 c + (1-2\mu)^2}}{c} ] & \text{ in the   Earth component} \\ x \in [1, \frac{ -1-\sqrt{  c^2 + 2(1-2\mu)c +1}}{c}], \; y \in [ \frac{ -(1-2\mu) - \sqrt{ c^2 +2 c + (1-2\mu)^2}}{c} , 1] & \text{ in the Moon component}                              \end{cases}
\end{equation}
see \cite{Bifurcation}, \cite{Kim}.

\begin{Lemma} There exist $a=a(\mu)$,  $b=b(\mu) \in \R$, $-1<a< 0 <b<1$, having the property that $A>0$ for any $x$ and for $y \in [-1, a] \cup [b, 1]$. 
\end{Lemma}
\begin{proof} We first claim that $2cx^2 + x - c <0$. To see this, we observe that it has two real roots since the discriminant is positive: $1+8c^2 >0$.  By Vieta's formulas   the product of the two roots is negative. Moreover, plugging $x=1$ into $2cx^2 + x - c$   gives rise to $c+1<0$. Since the leading coefficient $2c$ is negative, we then conclude that the positive root is less than $1$. This proves the claim.

Abbreviate by $f(y) = 2cy^2 + (1-2\mu)y - c$. We compute that
\begin{eqnarray*}
f(-1) &=& 2c - (1-2\mu) -c = c - (1-2\mu) <0\\
f(0) &=& -c >0\\
f(1) &=& 2c + (1-2\mu) - c = c+(1-2\mu) <0.
\end{eqnarray*}
This implies that $f$ has two real roots $a,b$ such that $-1<a<0<b<1$ and satisfies $f<0$ for $y \in [-1, a) \cup (b, 1]$ and $f>0$ for $a<y<b$. On the other hand, for $c<c_J$ the term  $1-2\mu + cy $  never equals  zero in the domain (\ref{ellipticdomain}). 

To prove the lemma, it now suffices to show that  at least one term among $cx^2 + 2x - cy^2 - 2(1-2\mu)y$, $1-y^2$, $x^2-1$ and $1+cx$ does not equal zero. Indeed, that $x^2 -1 = 1+cx = 0$ implies that $c=-1$, which contradicts to $c<c_J$.  This finishes the proof of the lemma.
\end{proof}

As a direct consequence, we obtain 
\begin{Corollary} \label{cisrcor}There exists $c_E < c_J$ (or $c_M < c_J$) such that for $c<c_E$ (or $c<c_M$) ,  the regularized Earth (or Moon) component is convex.
\end{Corollary}
\begin{proof} By (\ref{ellipticdomain}) and the previous lemma, we just need to define $c_E$ and $c_M$ by energy values satisfying
\begin{equation*}
a= \frac{ -(1-2\mu) + \sqrt{ c_E^2 +2 c_E + (1-2\mu)^2}}{c_E}
\end{equation*}
and
\begin{equation*}
b=\frac{ -(1-2\mu) - \sqrt{ c_M^2 +2 c_M + (1-2\mu)^2}}{c_M}.
\end{equation*}
\end{proof}

In what follows, we assume that $y \in (a,b)$, provided that $c$ is sufficiently large. For a fixed $x$, we abbreviate by $A_x(y) := A(x,y)$ and differentiate
\begin{equation*}
A_x'(y) = ( 1-2\mu + 4cy)( - c(2cx^2 + x - c)y^2 - 2( 2cx^2 + x - c)(1-2\mu)y + c^2x^4 + 3cx^3 + x^2 +1 ).
\end{equation*}
We assume that $ - c(2cx^2 + x - c)y^2 - 2( 2cx^2 + x - c)(1-2\mu)y + c^2x^4 + 3cx^3 + x^2 +1 $ admits a real root $y_0 = y_0(x)$. Using this we obtain that
\begin{equation*}
A_x (y_0) = (2cx^2 + x - c)( y_0^2 -1)( cy_0 + 1-2\mu)^2.
\end{equation*}
Since $cy + 1-2\mu$ is nonvanishing by  (\ref{ellipticdomain}) and $ y \in (a,b)$,   we conclude that $A_x(y_0) >0$.

 We now consider the remaining root $y = -(1-2\mu)/4c$ of $A'_x(y)$ and observe that 
\begin{equation*}
A_x \bigg(\frac{1-2\mu}{-4c} \bigg) = -\bigg( \frac{(1-2\mu)^2}{ 8c} + c\bigg) \bigg( c^2 x^4 + 3cx^3 + x^2 +1\bigg) - \bigg(2cx^2 + x - c\bigg)\bigg( \frac{ 5(1-2\mu)^4 }{256c^2} + (1-2\mu)^2\bigg).
\end{equation*}
It follows that $A_x( -(1-2\mu)/4c)>0$ if $c^2 x^4 + 3cx^3 + x^2 +1>0$. Assume that $c^2 x^4 + 3cx^3 + x^2 +1 =0$. If $ \mu \neq 1/2$, it is obvious that $A_x( -(1-2\mu)/4c) >0$. If $\mu = 1/2$, then we have $ A_x(-(1-2\mu)/4c ) = A_x(0) = 0$. Recall that  the variable $y$ varies in $ [-1, 0)$ for the Earth component and in $ (0, 1]$ for the Moon component. Then by means of the fact that $A_x$ is a quartic polynomial of $y$ whose leading coefficient is $-c^2 ( 2cx^2 + x - c) >0$, the arguments so far prove the following proposition which improves Corollary \ref{cisrcor}.

\begin{Proposition}\label{oriosgesd} There exists $c_E' \in (c_E, c_J)$ (or $c_M' \in (c_M, c_J)$)  such that the assertion of Corollary \ref{cisrcor} holds true. In particular, we have $c_E'=c_M'=c_J =-2$ for $\mu=1/2$ and hence if the two primaries have the equal mass, then each compact component bounds a strictly convex region for any energy below the critical Jacobi energy.
\end{Proposition}

\;

In what follows, we may further assume that $c^2 x^4 + 3cx^3 + x^2 +1 <0$, see Figure \ref{fig:theorem3}. Since
\begin{eqnarray*}
A&=& \bigg( 2cy^2 + (1-2\mu)y -c \bigg)\bigg(  (cx^2 + 2x)  ( 2c x^2 +x -c) -  (x^2-1)(1+cx)^2   \bigg) \\
&& - \bigg( 2c x^2 +x -c \bigg)\bigg(  (cy^2 + 2(1-2\mu)y) ( 2cy^2 + (1-2\mu)y -c)    +(1-y^2)( 1-2\mu +cy)^2    \bigg)\\
&=& ( 2cy^2 + (1-2\mu)y -c ) (c^2 x^4 + 3cx^3 + x^2 +1) \\
&&- ( 2c x^2 +x -c )(c^2y^4 + 3c(1-2\mu)y^3 + (1-2\mu)^2 y^2 + (1-2\mu)^2 ),
\end{eqnarray*}
that $2c x^2 +x -c<0$, $2cy^2 + (1-2\mu)y -c >0$ and $c^2 x^4 + 3cx^3 + x^2 +1 <0$  implies that we may assume that $c^2y^4 + 3c(1-2\mu)y^3 + (1-2\mu)^2 y^2 + (1-2\mu)^2 >0$.

We now abbreviate by $A_y(x) = A(x,y)$  and differentiate it 
\begin{equation*}
A_y'(x)  = ( 1 + 4cx)\bigg( (2cy ^2 + (1-2\mu ) y -c)cx^2 + 2(2c y^2 +(1-2\mu)y -c ) x -(c^2y^4 + 3c(1-2\mu)y^3 + (1-2\mu)^2 y^2 + (1-2\mu)^2) \bigg). 
\end{equation*}
Assume that $$ g_y(x):= (2cy ^2 + (1-2\mu ) y -c)cx^2 + 2(2c y^2 +(1-2\mu)y -c ) x -(c^2y^4 + 3c(1-2\mu)y^3 + (1-2\mu)^2 y^2 + (1-2\mu)^2) $$ admits no real roots. Then $g_y(x) <0$ and since  $1+4cx<0$, we have $A_y'(x) >0$.
\begin{figure}[t]
 \centering
 \includegraphics[width=0.65\textwidth, clip]{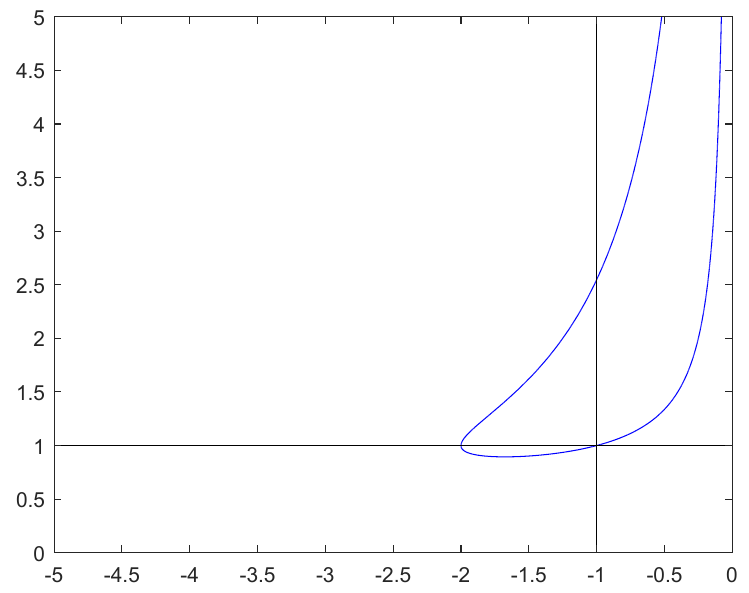}
  \caption{ The curve $c^2 x^4 +3cx^3 +x^2 +1=0$. The horizontal axis is the $x$-axis and the vertical one is the $c$-axis. In the region enclosed by the blue curve and the two lines $c=-1$ and $x=1$, we have $c^2 x^4 +3cx^3 +x^2 +1>0$. The leftmost point is $(x,c)=(1,-2)$.}   
 \label{fig:theorem3}
\end{figure}
Assume that $ g_y$ has a real root.  We  note that $g_y(0) <0$ and by the Vieta's formulas the sum of real roots equals 
\begin{equation*}
- \frac{ 2( 2cy^2 + (1-2\mu)y - c)}{ c( 2cy^2 + (1-2\mu)y - c)} = \frac{2}{-c} \in (0,2).
\end{equation*}
We claim that both the two roots are less than 1 in the Moon component. By the previous argument, it suffices to show that $ g_y(1) <0$. We abbreviate by
\begin{equation*}
h(y) := g_y(1)  = ( 2cy^2 + (1-2\mu)y - c)(c+2) - (c^2y^4 + 3c(1-2\mu)y^3 + (1-2\mu)^2 y^2 + (1-2\mu)^2).
\end{equation*}
and differentiate that
\begin{equation*}
h'(y) = (1-2\mu + 4cy) ( - cy^2 - 2(1-2\mu)y + c+2).
\end{equation*}
Note that $y_{\pm}:=( -(1-2\mu) \pm \sqrt{ c^2 + 2c+ (1-2\mu)^2 })/  c$, which are the boundary values for $y$ in the Earth and the Moon components, see (\ref{ellipticdomain}),  are   roots of $h'$. Using $-c y_{\pm}^2 -2(1-2\mu)y_{\pm}+c+2=0$ we compute that 
\begin{eqnarray*}
&&  c^2y_{\pm}^4 + 3c(1-2\mu)y_{\pm}^3 + (1-2\mu)^2 y_{\pm}^2 + (1-2\mu)^2\\
 &=&   c(1-2\mu)y_{\pm}^3 + c(c+2)y_{\pm}^2 + (1-2\mu)^2 y_{\pm}^2 + (1-2\mu)^2 \\
&=&  -2(1-2\mu)^2 y_{\pm}^2 + (1-2\mu)(c+2)y_{\pm}  + c(c+2)y_{\pm}^2 + (1-2\mu)^2 y_{\pm}^2 + (1-2\mu)^2 \\
&=&  - (1-2\mu)^2 y_{\pm}^2 + (1-2\mu)(c+2)y_{\pm}  + c(c+2)y_{\pm}^2   + (1-2\mu)^2 .
\end{eqnarray*}
which follows that
\begin{eqnarray*}
h(y_{\pm}) &=& c(c+2)y_{\pm}^2 - c(c+2) + (1-2\mu)^2 y_{\pm}^2 - (1-2\mu)^2 \\
&=& ( c^2 + 2c + (1-2\mu)^2 )(y_{\pm}^2 -1).  
\end{eqnarray*}
Since $a<y_{\pm}<b$ and $c<c_J  = -1 - 2\sqrt{\mu(1-\mu)}$ we conclude that $h(y_{\pm})<0$. For the remaining root $-(1-2\mu)/4c$ of $h'(y)$, we observe that
\begin{equation*}
\frac{  -(1-2\mu )}{ 4c} < y_- = \frac{ -(1-2\mu) - \sqrt{ c^2 + 2c+ (1-2\mu)^2 }}{ c} .
\end{equation*}
Thus, we conclude that $h(y)=g_y(1) <0$  in the Moon component, which proves the claim. 

On the other hand, we observe that
\begin{equation*}
 \frac{-(1-2\mu)+\sqrt{ c^2 + 2c +(1-2\mu)^2}}{c} < -\frac{1-2\mu}{4c} \;\;\; \Rightarrow  \;\;\; c < c_E'':=-1-\frac{ \sqrt{ -28\mu^2 + 28\mu + 9 }}{4} .
\end{equation*}
This shows that both the two roots of $g_y(x)$ are less than 1 in the Earth component if $c<c_E''$. 

We have shown that if either the satellite moves in the Earth component for $ c < c_E''$ or it moves in the Moon component, $ g_y(x) <0$ and hence $A_y'(x) >0$. As a result, $A_y$ is an increasing function and by
\begin{equation*}
A_y(1) = (c+1)g_x(1)>0   
\end{equation*}
we then conclude that $A_y(x) >0$. This gives rise to the following proposition which improves Proposition \ref{oriosgesd}.

\begin{Proposition}Assume that $ \mu< 1/2$. The regularized Moon component bounds a strictly convex domain. The same assertion holds true for the Earth component if $c < c_E''$ with $c_E''>c_E'$.
\end{Proposition}

\begin{figure}[t]
 \centering
 \includegraphics[width=0.65\textwidth, clip]{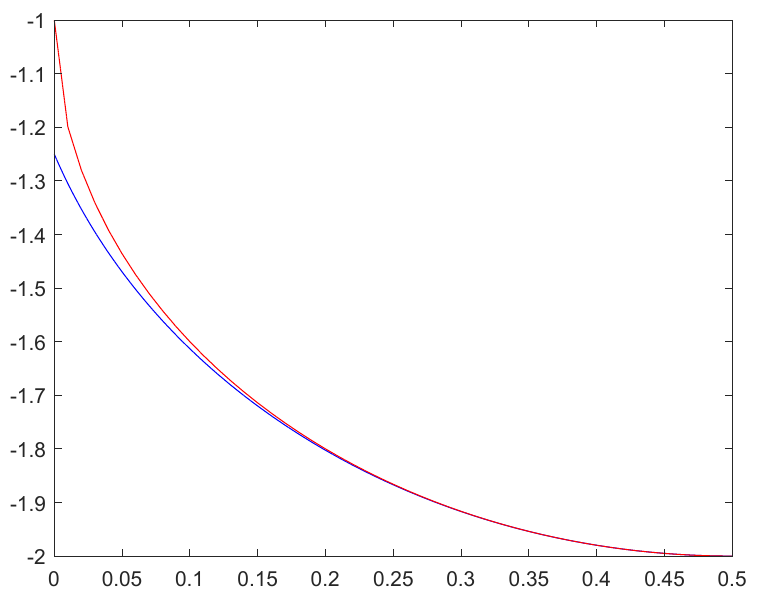}
  \caption{   The red curve is $c_J=c_J(\mu)$ and the blue one represents $c_0 = c_0(\mu).$  }
  \label{c0cmu}
\end{figure}

To complete the proof of the theorem, we consider the remaining case: the satellite is confined to a neighborhood of the Earth and $c_E'' \leq c <c_J$. In this case $h'(y)$ has three solutions
\begin{equation*}
- \frac{1-2\mu}{4c} \leq y_+ = \frac{ -(1-2\mu) + \sqrt{ c^2 + 2c +(1-2\mu)^2}}{c} < y_- = \frac{ -(1-2\mu) - \sqrt{ c^2 + 2c +(1-2\mu)^2}}{c}.
\end{equation*}
Recall that $h(y_{\pm}) <0$.  We observe that
\begin{equation*}
h\bigg( - \frac{1-2\mu}{4c} \bigg)= - \frac{1}{c^2} \eta(c),
\end{equation*} 
where $\eta(c) = c^4 + 2c^3 + (9/8)(1-2\mu)^2 c^2 +( (1-2\mu)^2/4)c + (5/256)(1-2\mu)^4$. We differentiate
\begin{eqnarray*}
\eta'(c) &=& 4c^3 + 6c^2 +\frac{9}{4} (1-2\mu)^2 c- \frac{(1-2\mu)^2}{4} \\
\eta''(c) &=& 12c^2 +12c + \frac{9}{4}(1-2\mu)^2.
\end{eqnarray*}
We observe that the discriminant of $\eta''$ is positive
\begin{equation*}
D/4(\eta'') = -108 \mu^2 + 108 \mu +9 >0. 
\end{equation*}
Moreover,  the sum of the real roots is negative and the product is positive. By means of $\eta''(0)>0$, we conclude that the two real roots are negative. Since
\begin{equation*}
\eta''(c_J) = ( -39 \mu^2 + 39\mu +\frac{9}{4} ) + 24\sqrt{\mu(1-\mu)}>0,
\end{equation*}
we conclude that $\eta'' >0$ for $c<c_J$, which implies that $\eta'$ is increasing for $c<c_J$. We then observe that
\begin{equation*}
\eta'(c_J)=(14\mu^2 - 14\mu - \frac{9}{2})\sqrt{\mu(1-\mu)} - 16\mu(1-\mu) <0
\end{equation*}
and hence $\eta$ is decreasing for $c<c_J$. We finally compute that
\begin{equation*}
\eta( c_J) = - \frac{27}{256}(1-2\mu)^4 <0
\end{equation*}
and
\begin{equation*}
\eta(c_E'') = \frac{9}{32}(1-2\mu)^2 ( \sqrt{ -28\mu^2 +28\mu+9} -4\mu^2 +4\mu +3)>0.
\end{equation*}
We conclude that there exists $c_0 \in ( c_E'', c_J)$ such that 
\begin{equation*}
\begin{cases} \displaystyle h\bigg( \frac{-(1-2\mu)}{4c} \bigg) <0 \text{ if } c<c_0 \\  \\ \displaystyle h\bigg( \frac{-(1-2\mu)}{4c} \bigg) \geq 0 \text{ if } c \geq c_0 . \end{cases}
\end{equation*}
By the previous argument, this implies that the Earth component bounds a strictly convex domain for $c<c_0$ and it fails to be strictly convex for $c \geq c_0$, see Figure \ref{c0cmu}.

\;

\section{Proof of Theorem \ref{theorem2}}\label{sec5}

To prove  Theorem \ref{theorem2},  we again use Theorem \ref{theoremsalomao}.

Consider the boundary $\partial \pi ( K_c^{-1}(0))$ along which $V$ equals zero, where $K_c$ is given in (\ref{eq:Levi}) and 
\begin{equation}\label{Levipotential}
V(x,y) :=V_c(x,y)= - c(x^2 +y^2) - \frac{\mu (x^2 +y^2)}{ \sqrt{ 4x^4 + 8x^2y^2 - 4x^2 + 4y^4 + 4y^2 +1   }} - \frac{1-\mu}{2}.
\end{equation}
Observe that the potential function $V$ also depends on the energy level.  Along the boundary, the left-hand side of the inequality (\ref{eq:convex})  becomes
\begin{equation}\label{theorem2proof}
F:=V_{xx}V_y^2 + V_{yy}V_x^2 - 2 V_x V_y V_{xy} . 
\end{equation}
In the following we show that for $c=c_J$ and for $\mu < 16/17$ the function $F$ fails to be positive near a critical point of $V_{c_J}$. By continuity this proves Theorem \ref{theorem2}.

\;
 
We first compute   critical points of $V$. 

\begin{Lemma}\label{lemma51} There exists a precisely three critical point $(0,0)$ and $ (\pm x_0 = \pm \sqrt{ (1/2)(1-\sqrt{\mu/-c} )}, 0)$ of $V$.    The two critical points $(\pm x_0, 0)$ lie on the curve $V_{c_J}=0$.
\end{Lemma}
\begin{proof} That $(0,0)$ is a critical point is straightforward from 
\begin{eqnarray*}
&&V_x = - \frac{2x A(x,y)}{ \sqrt{ 4x^4 - 8x^2 y^2 - 4x^2 + 4y^4 + 4y^2 +1}^3} ,\\
&&V_y = - \frac{ 2yB(x,y)}{\sqrt{ 4x^4 - 8x^2 y^2 - 4x^2 + 4y^4 + 4y^2 +1}^3} ,
\end{eqnarray*}
where
\begin{eqnarray*}
&&A(x,y) = c \sqrt{ 4x^4 - 8x^2 y^2 - 4x^2 + 4y^4 + 4y^2 +1}^3 - \mu ( 2x^2 - 6y^2 -1),\\
&&B(x,y) = c \sqrt{ 4x^4 - 8x^2 y^2 - 4x^2 + 4y^4 + 4y^2 +1}^3 - \mu ( 6x^2 - 2y^2 -1).
\end{eqnarray*}

To find another critical points, we first suppose that $x\neq 0$ and $y \neq 0$. For a point $(x,y)$ be a critical point of $V$, it then must satisfy $A(x,y)=B(x,y)=0$. Plugging  $A=0$ into $B=0$ gives rise to $-4\mu(x^2 + y^2 ) =0$, which implies that $ x = y = 0$. This contradicts the assumption.

We next suppose that $ x =0$ and $y \neq 0$. We observe that
\begin{eqnarray*}
V_y (0, y) = - \frac{2\mu}{ (1+2y^2 )^2} ( 4cy^4 + 4cy^2 + c + \mu).
\end{eqnarray*}
To check whether $V_y(0,y)=0$ admits a real root, we regard  $4cy^4 + 4cy^2 + c + \mu$ as a polynomial of $y^2$ and  see that the discriminant is positive:
\begin{equation*}
D/4 = ( 2c)^2 - 4c (c+\mu) = -4\mu c >0.
\end{equation*}
It then has two real roots and by the Vieta's formulas they are both negative. This implies that $V_y (0,y)=0$ admits no real roots. 

Finally, we assume that $ x\neq 0$ and $ y = 0$. As in the previous case, we observe 
\begin{equation*}
V_x (x,0) = \begin{cases} \displaystyle - \frac{2x}{ (2x^2 -1)^2} ( 4cx^4 - 4cx^2 + c + \mu) & \text{ if } 2x^2 -1 <0 \\  \displaystyle - \frac{2x}{ (2x^2 +1)^2} ( 4cx^4 - 4cx^2 + c - \mu) & \text{ if } 2x^2 -1 >0 \end{cases}
\end{equation*}
We then compute the discriminants of $  4cx^4 - 4cx^2 + c  \pm \mu$, which is regarded as a polynomial of $x^2$: 
\begin{equation*}
D/4= \mp 4c\mu.
\end{equation*}
It follows immediately that there are no critical points in the latter case. For the former case, we compute that
\begin{equation*}
x^2 = \frac{1}{2} \pm \frac{1}{2} \sqrt{ \frac{\mu}{-c}}.
\end{equation*}
Since $2x^2-1<0$, we conclude that $V_x(x,0)=0$ has precisely two real solutions
\begin{equation*}
\pm x_0 := \pm \sqrt{ \frac{1}{2} - \frac{1}{2} \sqrt{ \frac{\mu}{-c}}}.
\end{equation*}

To prove the last assertion, we need to show that $V_{c_J} ( \pm x_0, 0)=0$. This follows from 
\begin{equation*}
V( \pm x_0 , 0 )= \frac{  ( c + \sqrt{ - \mu c } + 1-\mu  ) \sqrt{ -\mu c} + \mu ( c + \sqrt{ - \mu c})}{ 2 \sqrt{ - \mu c}}
\end{equation*}
and
\begin{equation*}
( c + \sqrt{ - \mu c } + 1-\mu  ) \sqrt{ -\mu c} + \mu ( c + \sqrt{ - \mu c}) =0 \;\; \; \Leftrightarrow\;\; \; c = -1 \pm 2 \sqrt{  \mu (1-\mu) }.
\end{equation*}
This completes the proof of the lemma.
\end{proof}

By abuse of the notation we abbreviate by $V= V_{c_J}$. A direct consequence of the previous lemma  is that the curves $V=0$ and $F=0$ intersect at $(\pm x_0, 0)$, where $F$ is  defined as  in (\ref{theorem2proof}). In what follows, we may concentrate on $(x_0, 0)$.

\begin{Lemma}\label{lemma52} Both  curves $V=0$ and $F=0$ are tangent to the lines $y = \pm \sqrt{2}(x- x_0)$ at $( x_0, 0)$.
\end{Lemma}
\begin{proof} We differentiate $V(x,y)= 0$ twice with respect to $x$ and obtain
\begin{equation*}
V_{xx} + 2V{xy} \frac{dy}{dx} + V_{yy} \bigg( \frac{dy}{dx} \bigg)^2 + V_{y} \frac{ d^2 y}{dx^2} = 0.
\end{equation*}
Since $V_y(x_0, 0) = V_{xy}(x_0,0)=0$ we obtain $V_{xx} (x_0, 0) + V_{yy} (x_0, 0) ( dy/dx)^2 =0$ (cf. Lemma \ref{lemma31}). We compute that
\begin{equation*}
V_{xx}(x_0, 0) = -8 c_J \bigg( 1 - \sqrt{ \frac{ -c_J }{\mu} } \bigg)
\end{equation*}
and
\begin{equation*}
 V_{yy}(x_0, 0) =  4 c_J \bigg( 1 - \sqrt{ \frac{ -c_J }{\mu} } \bigg),
\end{equation*}
which follows that
\begin{equation*}
\bigg( \left. \begin{matrix} \displaystyle   \frac{dy}{dx} \end{matrix} \right|_{(x,y)=(x_0, 0)} \bigg) = 2.
\end{equation*}
Similarly, by $V_x (x_0, 0) = V_y (x_0, 0)=V_{xy} (x_0, 0) =0$ we obtain that
\begin{eqnarray*}
F_{xx}(x_0, 0) + 2 F_{yy}(x_0, 0 ) &=& 2V_{xx}(x_0, 0)^2 V_{yy}(x_0, 0) + 4 V_{xx}(x_0, 0) V_{yy}(x_0, 0)^2 \\
&=& 2V_{xx}(x_0, 0)V_{yy}(x_0, 0) ( V_{xx}(x_0, 0) + 2V_{yy}(x_0, 0) ) =0.
\end{eqnarray*}
This finishes the proof of the lemma.
\end{proof}

The next Lemma shows that $F>0$ at intersection points of the curve $V=0$ and the $x$-axis.

\begin{Lemma}\label{lemma53}   $F(x,0) \geq 0 $   and the equality holds if and only if $x = 0, \pm x_0$. 
\end{Lemma}
\begin{proof}  Plugging $y=0$ into the equation of $F$ gives rise to
\begin{equation*}
F(x, 0) = V_{x}(x, 0)^2 V_{yy}(x, 0).
\end{equation*} 
It suffices to show that $V_{yy}(x,0) >0$ for $x \neq 0, \pm x_0$. Since 
\begin{equation*}
V_{yy}(x,0) = -2c_J + \frac{ 2\mu ( 24x^6 - 28x^4 +10x^2-1)}{\sqrt{ 4x^4 - 4x^2+1}^5} =  -2c_J + 2\mu \frac{ 6x^2 -1}{|2x^2-1|^3}
 \end{equation*}
and  the function $(6x^2-1)/|2x^2-1|^3$ attains the minimum $-1$ at $x=0$, this finishes the proof of the lemma.
\end{proof}

We define the two functions $\widetilde{V}(x) :=V(x,  \sqrt{2} ( x - x_0) ) $ and $ \widetilde{F}(x) := F(x,   \sqrt{2} (x-x_0) ).$ The next step is to examine their derivatives.

\begin{Lemma}   $\widetilde{V}(x_0)=\widetilde{V}'(x_0 ) = \widetilde{V}''(x_0) = 0 $ for all $\mu$ and $\widetilde{V}'''(x_0) >0$ for $\mu < 16/17$. 
\end{Lemma}
\begin{proof} In view of  $V_{xy}(x_0,0)=V_{xxy}(x_0,0)=V_{yyy}(x_0,0) =0$, we have
\begin{eqnarray*}
\widetilde{V}''(x_0 )&=& V_{xx}(x_0, 0) + 2 V_{yy}(x_0, 0) = 0\\
\widetilde{V}'''(x_0 ) &=& V_{xxx}(x_0, 0) + 6V_{xyy}(x_0, 0) = 48 \mu \frac{ x_0 ( 10x_0^2 -1)}{ ( 2x_0 ^2 -1)^4} .
\end{eqnarray*} 
We now observe that
\begin{equation*}
 0 < 10 x_0^2 -1 = 4 - 5 \sqrt{ \frac{\mu}{-c_J}}  \;  \Leftrightarrow  \;    \mu   < \frac{ 16}{25}( 1 + 2 \sqrt{\mu(1-\mu)} ) .
\end{equation*}
If $\mu \leq 16/25$,  we are done. Assume that $\mu > 16/25$ and then 
\begin{equation*}
   \mu   < \frac{ 16}{25}( 1 + 2 \sqrt{\mu(1-\mu)} ) \; \Leftrightarrow \; (97\mu - 16)( 17\mu-16) <0 \; \Leftrightarrow \;  \mu < 16/17.
\end{equation*}
This finishes the proof of the lemma.
\end{proof}

\begin{figure}[t]
 \centering
 \includegraphics[width=0.6\textwidth, clip]{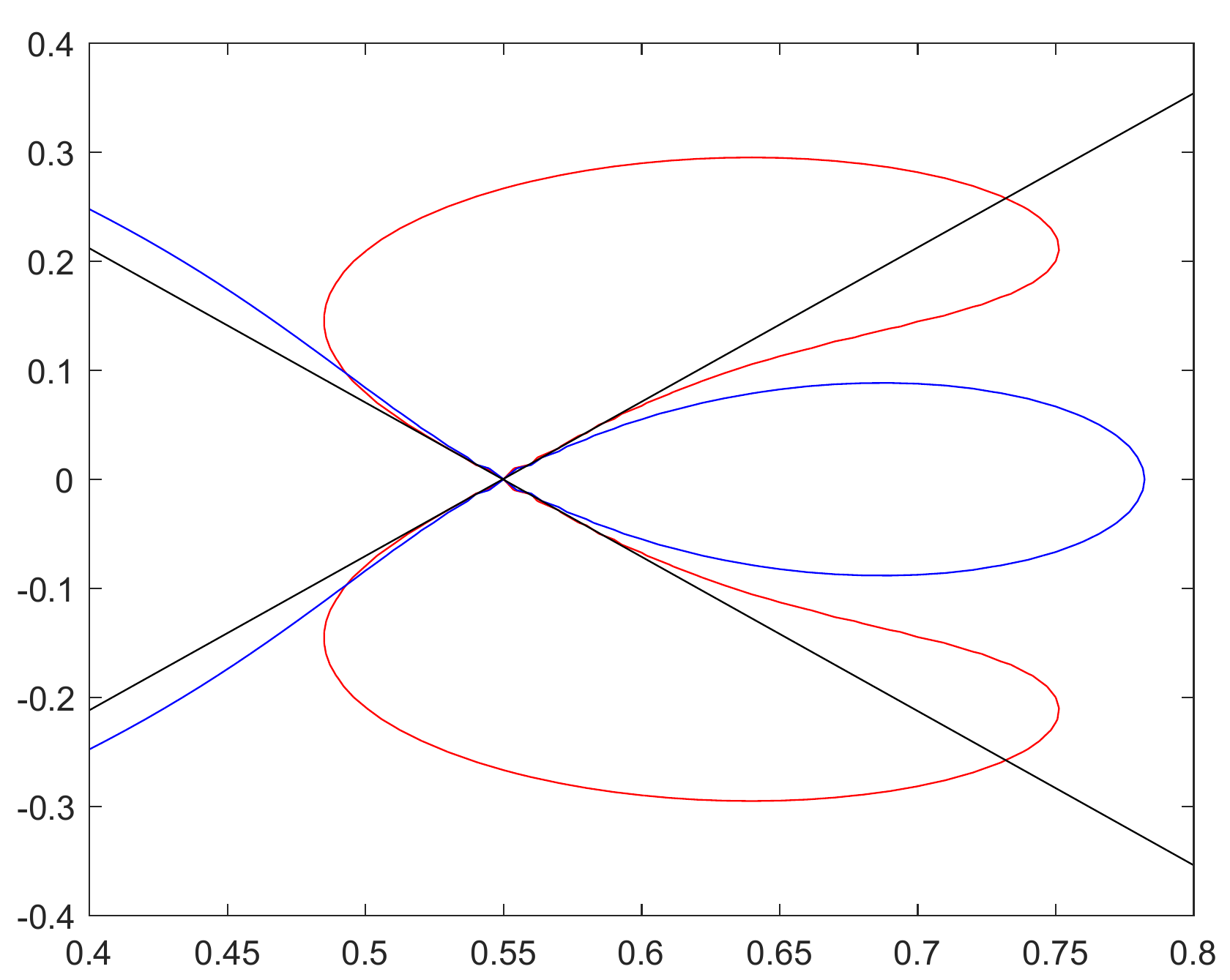}
  \caption{ Illustration for $\mu = 0.3$. The  blue curve is $V_{c_J} =0$ and the red one represents $F  =0$. They intersect at $(x_0,0)$ and are tangent to the black lines: $y =  \pm \sqrt{2} ( x-x_0).$}   
 \label{fig:theorem2}
\end{figure}

\begin{Lemma} $\widetilde{F}'(x_0 ) = \widetilde{F}''(x_0) = \widetilde{F}'''(x_0 ) =0.$
\end{Lemma}
\begin{proof} Similar to above, we compute that
\begin{eqnarray*}
\widetilde{F}''(x_0 )&=& F_{xx}(x_0, 0) + 2 F_{yy}(x_0, 0) = 0\\
\widetilde{F}'''(x_0 ) &=& F_{xxx}(x_0, 0) + 6F_{xyy}(x_0, 0) \\
&=& 6( V_{xx}(x_0, 0) + 2 V_{yy}(x_0, 0) )( V_{xx}(x_0, 0)V_{xyy}(x_0, 0) + V_{xxx}(x_0, 0)V_{yy}(x_0, 0) ) =0,
\end{eqnarray*} 
which complete the proof of the lemma.
\end{proof}

In view of the fact that the equations of $V$ and $F$ are symmetric with respect to the $x$-axis, the previous two lemmas give rise to the following corollary.
\begin{Corollary} There exists $\delta >0$ such that $F(x,y) <0$ for $(x,y) \in V^{-1}(0) \cap \left\{ (x,y) : x_0 - \delta < x < x_0 \right \}$, provided that $\mu < 16/17$. 
\end{Corollary}

By means of   Theorem \ref{theoremsalomao} and  Lemma \ref{lemma53}, this completes the proof of Theorem \ref{theorem2}, see Figure \ref{fig:theorem2}.

\;

\section{Proof of Theorem \ref{theorem1}}\label{sec4}
A strategy is similar to the one in the previous section.  Recall  that the curve $U(q) = c_J$ is homeomorphic to the figure eight whose vertex is $(l,0)$.  The first lemma we need considers tangent lines to $U(q)=c_J$ at $(l,0)$.

\begin{Lemma}\label{lemma31} The curve $U(q)= c_J$ is tangent to  two lines $q_2 =\pm \sqrt{2}(q_1-l)$ at $(l,0)$. 
\end{Lemma}
\begin{proof} Since $U_{q_1}(l,0) = U_{q_1 q_2}(l,0)=0$, we obtain $U_{q_1 q_1}(l,0) +U_{q_2q_2}(l,0)(dy/dx)^2 =0$ (cf. Lemma \ref{lemma52}). We compute that
\begin{equation*}
\bigg( \left. \begin{matrix} \displaystyle \frac{dq_2}{dq_1} \end{matrix} \right|_{(q_1, q_2) = (l,0)} \bigg)^2 = - \frac{ U_{q_1 q_1}(l,0)}{U_{q_2q_2}(l,0)} = - \frac{ - \frac{2(1-\mu)}{l^3} - \frac{ 2 \mu}{(1-l)^3}}{   \frac{ (1-\mu)}{l^3} + \frac{   \mu}{(1-l)^3} }= 2         .
\end{equation*}
This finishes the proof of the lemma.
\end{proof}

We define
\begin{equation}\label{functionV}
V(q_1):=U( q_1,  \sqrt{2}(q_1-l) ) -c_J = - \frac{1-\mu}{\sqrt{ q_1^2 + 2(q_1 - l)^2} } - \frac{\mu}{\sqrt{ (q_1-1)^2+ 2(q_1-l)^2}} -c_J.
\end{equation}
The second lemma is

\begin{Lemma} \label{lemmaderivatievedd} $\bullet$ For $\mu <1/2$, $V(l)=V'(l) =V''(l)=0$ and $V'''(l)>0$.

\hspace{1.7cm} $\bullet$ For $\mu =1/2$, $V(l)=V'(l) =V''(l)=V'''(l)=0$ and $V''''(l)>0$.
\end{Lemma}
\begin{proof} From 
\begin{eqnarray*}
V'(q_1) &=& \frac{(1-\mu)(3q_1 - 2l)}{\sqrt{ q_1^2 + 2(q_1-l)^2 }^3} +\frac{ \mu (3q_1 -1 - 2l)}{\sqrt{ (q_1-1)^2 + 2(q_1-l)^2 }^3} \\
V''(q_1)&=& 6(q_1-l)\bigg(  \frac{(1-\mu)(l-3q_1)}{\sqrt{ q_1^2 + 2(q_1-l)^2 }^5} +\frac{ \mu (l-3q_1 +2)}{\sqrt{ (q_1-1)^2 + 2(q_1-l)^2 }^5} \bigg) \\
V'''(q_1)&=& -6\bigg(  \frac{(1-\mu)(  l^2 - 12lq_1 + 9q_1^2      )(2l-3q_1)}{\sqrt{ q_1^2 + 2(q_1-l)^2 }^7} +\frac{ \mu ( l^2 - 12lq_1 + 9q_1^2 + 10l - 6q_1 -2)(2l-3q_1+1)  }{\sqrt{ (q_1-1)^2 + 2(q_1-l)^2 }^7} \bigg) \\
V''''(q_1) &=& 12a(q_1, l) \bigg( \frac{ 1-\mu }{\sqrt{ q_1^2 + 2(q_1-l)^2 }^9} + \frac{\mu}{\sqrt{ (q_1-1)^2 + 2(q_1-l)^2 }^9}\bigg)\\
&& + \frac{ 12 \mu (  -100l^3 + 504l^2 q_1 - 648 lq_1^2 + 216q_1^3 - 102l^2 + 144lq_1 + 20l - 48q_1 +7    ) }{\sqrt{ (q_1-1)^2 + 2(q_1-l)^2 }^9}, 
\end{eqnarray*}
where 
\begin{equation*}
a(q_1, l) = 13l^4 + 48l^3 q_1 - 324l^2 q_1^2 + 432l q_1^3 - 162 q_1^4,
\end{equation*}
that $V(l)=V'(l)=V''(l)=0$ is straightforward. We then observe that
\begin{eqnarray*}
V'''(l)&=& - \frac{ 6(1-\mu)(-2l^2)(-l)}{l^7} - \frac{6\mu(1-l)(-2)(1-l)^2}{(1-l)^7}\\
&=& - \frac{12(1-\mu)}{l^4 } + \frac{ 12 \mu}{ (1-l)^4} \\
&=& \frac{12(2l-1)}{l^2 (1-l)^2 ( 2l^2-2l+1)},
\end{eqnarray*}
where in the last step we used the relation $(1-\mu)(1-l)^2 = \mu l^2$. Since $l=l(\mu)$ is a decreasing function with $l(1/2)=1/2$, this proves the assertions for the third derivative. Finally, we see that $V''''(1/2) = 2688$ and  this completes the proof of the lemma.
\end{proof}

\begin{Corollary} \label{maincor}  There exists $\epsilon >0$ such that $V(q_1) <0$ for $q_1 \in (l-\epsilon, l)$, provided that $\mu <1/2$.
\end{Corollary}

The following corollary shows that if $\mu \geq 1/2$, i.e., either the Moon is heavier than the Earth or the two primaries have the equal mass,  then the Earth component $\mathcal{K}_c^E$ lies between  the two tangent lines given in Lemma \ref{lemma31}. 

\begin{Corollary}\label{propfiberwise} For $\mu \geq 1/2$, it holds that 
\begin{equation}
-\sqrt{2}(q_1 - l) < q_2 < \sqrt{2}(q_1 - l)
\end{equation}
for any $(q_1, q_2) \in \mathcal{K}_c^E$, provided that $c<c_J$.
\end{Corollary}
\begin{proof} We abbreviate by $\mathcal{K}_{c, \mu}^E$ the Earth component corresponding to the energy level $c$ and the mass ratio $\mu$. Given $\mu_1$, we translate $\mathcal{K}_{c_J, \mu_1}^E$ so that its vertex becomes $(1/2, 0)$. Then we note that $\mathcal{K}_{c, \mu_1}^E \subset \mathcal{K}_{-2, 1/2}^E$ for any $\mu > 1/2$. Since the assertion of Lemma \ref{lemma31} holds true for any $\mu$, this implies that   it suffices to prove  for $(q_1, q_2) \in \mathcal{K}_{-2, 1/2}^E$.

Suppose that $\mu = 1/2$ and hence $l=1/2$. We differentiate $V$ 
\begin{equation*}
 V'(q_1) =   \frac{ 3q_1 - 1}{2\sqrt{ q_1^2 + 2(q_1 - 1/2)^2 }^3}+\frac{ 3q_1 - 2}{2\sqrt{ (q_1-1)^2 + 2(q_1 - 1/2)^2 }^3}.
\end{equation*}
We claim that $V$ admits no critical points. Assume that $q_1$ is a critical point. Since $q_1 <1/2$, we see that $3q_1 -2<0$ and hence the term $3q_1 -1$ must be positive. Suppose that $q_1 >1/3$. Then we obtain
\begin{equation*}
V'(q_1) =0 \;\;\;  \Leftrightarrow \;\;\; \frac{1}{8}(1-2q_1)^2 (324q_1^4 - 648q_1^3 + 504q_1^2 - 180q_1^2 +23)=0.
\end{equation*}
One can easily see that  the term $324q_1^4 - 648q_1^3 + 504q_1^2 - 180q_1 +23$ is nonvanishing for $q_1 \in (1/3, 1/2)$ and this proves the claim. The corollary now follows immediately from Lemma \ref{lemmaderivatievedd}.  
\end{proof}

To check whether the curve $U(q) = c_J$ bounds a strictly convex domain or not, we examine   its curvature. The curvature is given by
\begin{equation}
\kappa =  \frac{  U_{q_1q_1} U_{q_2}^2 + U_{q_1}^2U_{q_2q_2}- 2U_{q_1q_2}U_{q_1}U_{q_2}  }{ \sqrt{ U_{q_1}^2 + U_{q_2}^2}^3} ,
\end{equation}
where it has a unique singularity at $(l,0)$. In what follows we consider the numerator, which is abbreviated by
\begin{equation}\label{cuvaturedbueartr}
C:=  U_{q_1q_1} U_{q_2}^2 + U_{q_1}^2U_{q_2q_2}- 2U_{q_1q_2}U_{q_1}U_{q_2}.
\end{equation}
Note that
\begin{equation*}
\left \{ (q_1, q_2)  \in \R^2 : C(q_1, q_2) = 0 \right \} = \left \{ (q_1, q_2 ) \in \R^2 \setminus \left \{ (l,0) \right \} :  \kappa =0 \right \} \cup \left \{ (l,0) \right \}. 
\end{equation*}
The function $C$ is explicitly given by
\begin{equation} \label{nusdoihs}
C(q) =    \frac{(1-\mu)^3}{r_1^7}  + \frac{\mu^3}{r_2^7} + \frac{ \mu(1-\mu)^2 (f(q) +r_2^2 g(q)) }{r_1^6r_2^5} + \frac{ \mu^2(1-\mu)( f(q )+r_1^2 g(q))}{ r_1^5r_2^6}  ,
\end{equation}
where  $r_1 = |q-E|$,  $r_2 = |q-M|$, $ f(q_1,q_2) = q_1^4 - 2q_1^3 + 2q_1^2q_2^2 + q_1^2 - 2q_1q_2^2 + q_2^4 - 2q_2^2 $ and $g(q_1, q_2) = 2q_1^2 - 2q_1 + 2q_2^2$.    From 
$$ f(q_1, q_2) = 0 \;\;\; \Longleftrightarrow\;\;\; \bigg(q_1-\frac{1}{2}\bigg)^2 + \bigg( q_2 \pm \frac{1}{\sqrt{2}}\bigg)^2 = \frac{3}{4}$$
and 
$$g(q_1, q_2) = 0 \;\;\; \Longleftrightarrow\;\;\; \bigg(q_1-\frac{1}{2}\bigg)^2 +  q_2^2 = \frac{1}{4} $$
we see that the curve $C=0$ must lie in the bounded region
\begin{equation}\label{regionpositieC}
\left \{ (q_1, q_2) : \bigg(q_1-\frac{1}{2}\bigg)^2 + \bigg( q_2 - \frac{1}{\sqrt{2}}\bigg)^2 < \frac{3}{4} \right \} \cup \left \{ (q_1, q_2) : \bigg(q_1-\frac{1}{2}\bigg)^2 + \bigg( q_2 + \frac{1}{\sqrt{2}}\bigg)^2 < \frac{3}{4} \right \}.
\end{equation}

\begin{Lemma}\label{curvaturpositive} The function $\widetilde{C}(q_1) := C(q_1, 0)$ has two singularities at $q_1=0, $ $1$. Moreover, $\widetilde{C} \geq 0$ and the equality holds if and only if $q_1 = l$. 
\end{Lemma}
\begin{proof} Plugging $q_2=0$ in (\ref{cuvaturedbueartr}) gives rise to
\begin{equation*}
C(q_1, 0) = U_{q_1}(q_1, 0)^2 U_{q_2q_2}(q_1, 0)=\bigg( \frac{ (1-\mu) q_1}{ |q_1|^2 } + \frac{\mu(q_1 - 1)}{|q_1-1|^3} \bigg)^2 \bigg( \frac{1-\mu}{|q_1|^3} +\frac{\mu}{|q_1 - 1|^3 } \bigg) \geq 0.
\end{equation*} 
Since $(l,0)$ is the only critical point of $U$, this completes the proof of the lemma.
\end{proof}

The curve $C=0$ might be comprised of several connected component. In view of the previous lemma we concentrate on the connected component, denoted by $\Gamma$, which intersects the $q_1$-axis at $(l,0)$. Note that in the region bounded by $\Gamma$ we have $C<0$ and hence $\kappa <0$. 

\begin{Lemma}\label{lemma32} 
The connected component $\Gamma$ is tangent to the two lines $q_ 2 = \pm \sqrt{2}(q_1 - l)$ at $(l,0)$.
\end{Lemma}
\begin{proof} Similar to Lemma \ref{lemma31},  we have
\begin{equation*}
\bigg( \left. \begin{matrix} \displaystyle \frac{dq_2}{dq_1} \end{matrix} \right|_{(q_1, q_2) = (l,0)} \bigg)^2 = - \frac{ C_{q_1 q_1}(l,0)}{C_{q_2 q_2}(l,0)}.
\end{equation*}
By means of the fact that $(l,0)$ is the critical point of $U$, we observe that
\begin{equation*}
C_{q_1 q_1}(l,0) = 2U_{q_1 q_1}(l,0)^2 U_{q_2q_2}(l,0), \;\;\;\;\; C_{q_2 q_2}(l,0) = 2U_{q_1 q_1}(l,0)  U_{q_2q_2}(l,0)^2. 
\end{equation*}
We then conclude that
\begin{equation*}
\bigg( \left. \begin{matrix} \displaystyle \frac{dq_2}{dq_1} \end{matrix} \right|_{(q_1, q_2) = (l,0)} \bigg)^2 = - \frac{ C_{q_1 q_1}(l,0)}{C_{q_2q_2}(l,0)}  = - \frac{ U_{q_1 q_1}(l,0)}{U_{q_2q_2}(l,0)}=2.
\end{equation*}
This finishes the proof of the lemma.
\end{proof}
 
Therefore, the two curves $U=c_J$ and $\Gamma$ are tangent to the same lines at $(l,0)$. Similar to (\ref{functionV}) we define the function
\begin{equation*}
C_l(q_1) = C(q_1,  \sqrt{2}(q_1- l) ).
\end{equation*}

\begin{Lemma}\label{lemmaCdd} $C(l)= C'_l(l)=C''_l(l)=C'''_l(l)=0$.
\end{Lemma}
\begin{proof} Since $U_{q_1}(l,0) = U_{q_2} (l,0)=0$, the first two assertions are straightforward. 

We  compute that
\begin{eqnarray*}
C_l '' (l) &=& C_{q_1 q_1} (l,0) + 2\sqrt{2} C_{q_1 q_2}(l,0) + 2 C_{q_2 q_2 } (l,0)\\
C_l'''(l) &=& C_{q_1q_1q_1}(l,0) + 3 \sqrt{2} C_{q_1 q_1 q_2 }(l,0) + 6 C_{q_1 q_2 q_2 }(l,0) + 2\sqrt{2} C_{q_2 q_2 q_2}(l,0).
\end{eqnarray*}
By means of $U_{q_1}(l,0) = U_{q_2}(l,0) = U_{q_1 q_2}(l,0) = 0$, 
we then observe that 
\begin{eqnarray*}
C_{q_1 q_2}(l,0) &=& C_{q_1 q_1 q_2}(l,0) = C_{q_2 q_2 q_2 }(l,0) = 0\\
C_{q_1q_1}(l,0) &=& 2U_{q_1 q_1}(l,0)^2 U_{q_2 q_2}(l,0)\\
C_{q_2 q_2 }(l,0) &=& 2U_{q_1 q_1}(l,0) U_{q_2 q_2}(l,0)^2 \\
C_{q_1 q_1 q_2}(l,0) &=& 2 U_{q_1 q_1}(l,0) U_{q_2 q_2}(l,0) U_{q_1 q_1 q_2}(l,0) + 2U_{q_1 q_1}(l,0)^2U_{q_2 q_2 q_2}(l,0)\\
C_{ q_2 q_2 q_2}(l,0)&=& 6U_{q_1 q_1}(l,0) U_{q_2 q_2}(l,0) U_{q_2 q_2q_2}(l,0)+6U_{q_2 q_2}(l,0)^2 U_{q_1 q_1 q_2}(l,0) .
\end{eqnarray*}
From  
\begin{equation*}
U_{q_1 q_1 } (l,0) = -2 U_{q_2 q_2}(l,0) = -2 \bigg( \frac{ 1- \mu}{l^3} + \frac{\mu}{(1-l)^3 } \bigg)
\end{equation*}
we now conclude that
\begin{equation*}
C_l ''(l) = 2U_{q_1 q_1}(l,0) U_{q_2 q_2}(l,0) \bigg( U_{q_1 q_1}(l,0) +2U_{q_2q_2}(l,0) \bigg)=0
\end{equation*}
and
\begin{equation*}
C_l'''(l) = 6\bigg( U_{q_1 q_1 q_1}(l,0)U_{q_2 q_2}(l,0) +U_{q_1 q_1}(l,0) U_{q_1 q_2 q_2}(l,0) \bigg) \bigg(  U_{q_1q_1}(l,0) + 2U_{q_2q_2}(l,0) \bigg)=0.
\end{equation*}
This completes the proof of the lemma.
\end{proof}

Since the equations of $U$ and $C$ are symmetric with respect to the $q_1$-axis, together with Corollary \ref{maincor}, the previous lemmas imply 
\begin{Corollary} There exists $\delta >0$ such that $C(q_1, q_2) <0$ for $(q_1, q_2) \in \left \{ U(q_1, q_2) = c_J \right \} \cap \left \{ (q_1, q_2) : l-\delta < q_1 < l \right \}$, provided that $\mu < 1/2$. 
\end{Corollary}

Since the curvature is positive  along $U=c_J$ at least outside of the region (\ref{regionpositieC}), this proves the assertion of  Theorem \ref{theorem1} for $\mu <1/2$. 

To finish the proof, we assume that $\mu = 1/2$. We now have
\begin{eqnarray*}
C_0 (q_1) &=& C(q_1, \pm \sqrt{2} (q_1 - 1/2) ) \\
&=& \frac{1}{8} \bigg( \frac{1}{r_1^7} + \frac{1}{r_2^7} + \frac{ f + r_2^2 g}{r_1^6 r_2^5} + \frac{ f + r_1^2 g}{r_1^5 r_2^6} \bigg) \\
&=& - 864(1-2q_1) \frac{ a(q_1)\sqrt{ 12q_1^2 - 8q_1 +2 }   + b(q_1)\sqrt{12 q_1 ^2 - 16q_1 +6 }           }{ (6q_1^2 - 8q_1+3)^3 ( 6q_1^2 - 4q_1+1)^3 \sqrt{ 12q_1^2 -16q_1+6} \sqrt{ 12q_1^2 - 8q_1+2}},
\end{eqnarray*}
where
\begin{equation*}
a(q_1) = q_1^5 - \frac{8}{3}q_1^4 + \frac{53}{18}q_1^3 - \frac{169}{108}q_1^2 + \frac{10}{27}q_1 - \frac{5}{216}
\end{equation*}
and
\begin{equation*}
b(q_1) = q_1^5 - \frac{7}{3}q_1^4 + \frac{41}{18}q_1^3 - \frac{137}{108}q_1^2 + \frac{11}{27}q_1 - \frac{13}{216}.
\end{equation*}
To show $C_0$ has no roots for $q_1 < 1/2$, we observe that
\begin{eqnarray*}
C_0(q_1) =0 \;\;\; &\Leftrightarrow& \;\;\; \sqrt{ 12q_1^2 - 8q_1 +2 } a(q_1) = - \sqrt{12 q_1 ^2 - 16q_1 +6 } b(q_1)\\
&\Rightarrow& \;\;\; ( 12q_1^2 - 8q_1 +2 ) a(q_1)^2 - (12 q_1 ^2 - 16q_1 +6 ) b(q_1)^2 =0\\
&\Rightarrow& \;\;\;  \frac{ (1-2q_1)^3}{11664} ( 7776 q_1^6 - 23328q_1^5 + 30348 q_1^4 - 21816 q_1^3 + 9232q_1^2 - 2212q_1 + 241)=0.
\end{eqnarray*}
Since $7776 q_1^6 - 23328q_1^5 + 30348 q_1^4 - 21816 q_1^3 + 9232q_1^2 - 2212q_1 + 241$ is positive, $C_0$ has no roots. Recall that the curve $C=0$ is bounded and $C>0$ for sufficiently large $|q|$. We have proven
\begin{Lemma} \label{lemmaposit} $C_0(q_1) >0$ for $ q_1 <1/2$.
\end{Lemma}

\begin{Corollary} \label{connectednbs}The connected component $\Gamma$ of the curve $C=0$, which passes through the point $(1/2, 0)$, lies in the region
\begin{equation*}
\left \{ (q_1, q_2) : q_2 \geq \pm \sqrt{2} (q_1 - 1/2) \right \}.
\end{equation*}
\end{Corollary}

If the curve $C=0$ is connected, then Corollary \ref{propfiberwise} and Corollary \ref{connectednbs} prove the theorem. 
 
We now suppose that $C=0$ comprises of several connected components and introduce the   polar coordinates $(q_1, q_2 )= ( r \cos \theta,  r \sin \theta)$. Then $C$ becomes
\begin{eqnarray*}
C(r, \theta) &=&  \frac{1}{ 8 r^7 \sqrt{ r^2 -2r \cos \theta +1 }^7} \bigg(  r^7 + \sqrt{ r^2 -2r \cos \theta +1 }^7 \\
&& \hspace{3.8cm}+   r^2 \sqrt{ r^2 -2r \cos \theta +1} ( 3r^4 -4r^3 \cos \theta +3r^2 \cos^2 \theta -2r^2)\\
&& \hspace{3.8cm}+  r(r^2 -2r \cos \theta +1)( 3r^4 -8r^3 \cos \theta +7r^2 \cos^2 \theta -2r\cos\theta) \bigg).
\end{eqnarray*}

We first consider its $r$-derivative.
\begin{Lemma}\label{rderi} The derivative $\partial_r C$ does not vanish for $ r < 1/2$. 
\end{Lemma}
\begin{proof} We write $x=r$ and $y=\cos \theta$ and define the function  $F: (0,1/2) \times [-1,1] \rightarrow \R$ by 
\begin{eqnarray*}
F(x,y)&:=& \bigg(  4 x^4 y^4   -  ( \frac{65}{7}x^5 + 8 x^3 )y^3   + ( \frac{235}{28} x^6 + \frac{345}{28}x^4 + 6x^2 ) y^2   \\
&&\hspace{0.3cm}- ( 4x^7 + \frac{38}{7}x^5 + 6x^3 + 2x) y + ( x^8 + \frac{13}{28}x^6 + \frac{9}{7}x^4 + x^2 + \frac{1}{4}) \bigg) \sqrt{ x^2 - 2xy+1 } \\
&&\hspace{0.3cm}+ x^2 \bigg(   \frac{13}{2}x^3 y^4  - ( \frac{393}{28}x^4 + \frac{207}{28}x^2) y^3  + ( \frac{333}{28}x^5 + \frac{297}{28}x^3 + \frac{39}{14}x ) y ^2  \\
&&\hspace{1.3cm} - (   5x^6 +\frac{21}{4}x^4 +  \frac{27}{14}x^2    + \frac{5}{14} )y + ( x^7 + \frac{27}{28}x^5 + \frac{3}{14}x^3 )     \bigg) 
\end{eqnarray*}
so that $ \partial_x C = -7 F (x,y) \big/ 2x^8 \sqrt{x^2 -2xy+1}^9$. If $F$ vanishes at $(x,y)$, then it satisfies

\begin{eqnarray*}
F_0(x,y)&:=&( x^2 - 2xy+1 )\bigg(  4 x^4 y^4   -  ( \frac{65}{7}x^5 + 8 x^3 )y^3   + ( \frac{235}{28} x^6 + \frac{345}{28}x^4 + 6x^2 ) y^2   \\
&&\hspace{2.5cm} - ( 4x^7 + \frac{38}{7}x^5 + 6x^3 + 2x) y + ( x^8 + \frac{13}{28}x^6 + \frac{9}{7}x^4 + x^2 + \frac{1}{4}) \bigg)^2  \\
&&- x^4 \bigg(   \frac{13}{2}x^3 y^4  - ( \frac{393}{28}x^4 + \frac{207}{28}x^2) y^3  + ( \frac{333}{28}x^5 + \frac{297}{28}x^3 + \frac{39}{14}x ) y ^2  \\
&& \hspace{1cm}- (   5x^6 +\frac{21}{4}x^4 +  \frac{27}{14}x^2    + \frac{5}{14} )y + ( x^7 + \frac{27}{28}x^5 + \frac{3}{14}x^3 )     \bigg) ^2 \\
&=& -32 x^9 y^9 + \bigg(  \frac{3425}{28}x^{10} + 144 x^8      \bigg) y^8 -  \bigg(    \frac{38917}{196}x^{11} + \frac{15021}{28}x^9 + 288 x^7    \bigg) y^7 \\
&&+  \bigg(     \frac{139155}{784} x^{12} + \frac{340323}{392}x^{10} + \frac{ 769431}{784}x^8 + 336x^6   \bigg)y^6 \\
&&-  \bigg(   \frac{ 4629}{49}x^{13} + \frac{39360}{49}x^{11} + \frac{19953}{14}x^9 + \frac{196145}{196}x^7 + 252x^5     \bigg) y^5 \\
&&+  \bigg(  \frac{411}{14}x^{14} + \frac{368931}{784} x^{12} + \frac{225543}{196}x^{10} + \frac{959533}{784}x^8 + \frac{30861}{49}x^6 +126x^4      \bigg) y^4\\
&& -  \bigg(   \frac{30}{7}x^{15} + \frac{8769}{49}x^{13} + \frac{111273}{196}x^{11} + \frac{77229}{98}x^9 + \frac{4293}{7}x^7 + \frac{49443}{196}x^5 + 42x^3     \bigg)y^3 \\
&&+ \bigg( \frac{288}{7}x^{14} +\frac{137229}{784}x^{12} + \frac{112503}{392}x^{10} + \frac{113205}{392}x^8 +\frac{17883}{98}x^6 +\frac{24709}{392}x^4 +9x^2  \bigg)y^2 \\
&&-  \bigg(   \frac{30}{7}x^{15} + \frac{1536}{49}x^{13} + \frac{5739}{98}x^{11} + \frac{1857}{28}x^9 + \frac{10869}{196}x^7 + \frac{ 211}{7}x^5 + 9x^3 +\frac{9}{8}x       \bigg)y \\
&&+ \bigg(        \frac{33}{14}x^{14} + \frac{4365}{784}x^{12} + \frac{1221}{196}x^{10} + \frac{2307}{392}x^8 + \frac{249}{56}x^6 + \frac{15}{7}x^4 + \frac{9}{16}x^2 +      \frac{1}{16}        \bigg) \\
&=&0.
\end{eqnarray*}
We regard $F_0$ as a polynomial of $y$ and abbreviate by $a_j(x)$ the  coefficient function of $y^j$ for each $j=0,1,\cdots, 9$. We differentiate $F_0$ with respect to $y$ seven times and obtain
\begin{equation*}
F_0^{(7)}(x,y) = \frac{9!}{2} a_9(x) y^2 + 8! a_8(x) y + 7! a_7(x).
\end{equation*}
We compute its discriminant 
\begin{equation*}
D( F_0^{(7)}) = (8!)^2 x^{18} \bigg( \frac{74647}{112} x^2 - \frac{23778}{7} \bigg).
\end{equation*}
Since $ x<1/2$, it is negative and hence $F_0^{(6)}$ is    decreasing. We then observe that
\begin{eqnarray*}
F_0^{(6)}(x,1) &=& \frac{9!}{6} a_9(x) + \frac{8!}{2}a_8(x) + 7! a_7(x) + 6!a_6(x) \\
&=& \frac{45}{49} x^6 ( 139155 x^6 - 1089676 x^5 + 3365846 x^4 - 5051508 x^3 + 3930519 x^2 - 1580544x + 263424)
\end{eqnarray*}
Again, since $x<1/2$, we have $ F_0^{(6)}(x,1) >0$ and then this implies that $F_0^{(5)}$ is  increasing. Proceeding in a similar way, we observe that for $x <1/2$ 

\begin{eqnarray*}
 F_0^{(5)}(x,1) <0 \;\; &\Rightarrow& \;\; \text{ $F_0^{(4)}$ is decreasing.} \\
 F_0^{(4)}(x,1) >0 \;\; &\Rightarrow& \;\; \text{ $F_0^{(3)}$ is increasing.} \\
 F_0^{(3)}(x,1) <0 \;\; &\Rightarrow& \;\; \text{ $F_0''$ is decreasing.} \\
 F_0''(x,1) >0 \;\; &\Rightarrow& \;\; \text{ $F_0'$ is increasing.} \\
 F_0'(x,1) <0 \;\; &\Rightarrow& \;\; \text{ $F_0 $ is decreasing.} 
\end{eqnarray*}
Finally we compute that
\begin{eqnarray*}
F_0(x,1) &=& \frac{1}{784}(1-2x)(1-x)^2(2x^2 - 2x +1)( 60x^4 -120x^3 +102x^2 - 42x+7)\\
&& \hspace{0.63cm}(28x^6 - 84x^6 + 150x^4 - 160x^3 +108x^2 -42x+7).
\end{eqnarray*}
One can easily see that all the terms are positive. This implies that $F_0>0$ and hence $F_0$ does not vanish. This completes the proof of the lemma.
\end{proof}

We prove a similar result for the $\theta$-derivative.

\begin{Lemma}\label{thetaderi} $\partial_{\theta} C$ does not vanish for $(r, \theta) \in (0,4/5]\times [\pi/2, \pi)$. Moreover, there exists $r_0 <1/2$ such that the same assertion holds for $(r,\theta) \in [r_0, 4/5] \times  [\theta_0, \pi/2]$. 
\end{Lemma}
\begin{proof} As in the proof of the previous lemma we plug $r=x, \cos\theta=y$ in the equation of $\partial_{\theta}C$ and define the function $G:(0,4/5) \times (-1,0] \rightarrow \R$ by
\begin{eqnarray*}
G(x,y)   &:=&    ( 6x^3 y^2   - (10x^4 - 6x^2 ) y + 14x^5 - 16x^3 ) \sqrt{ x^2 - 2xy+1} \\
&& - 14x^3 y^3  + (27x^4 - 9x^2 ) y^2   - (24x^5 - 18x^3 - 12x)y + 14x^6 - 3x^4 -12x^2 -2 
\end{eqnarray*}
so that $\partial_{\theta} C (x,y)  = -G (x,y) \sin\theta / 8x^5 \sqrt{ x^2 -2xy+1}^9$. We claim that  $G$ does not vanish, provided that the variables are given as in the assertion. We abbreviate by $a(x,y) = 6x^3 y^2   - (10x^4 - 6x^2 ) y + 14x^5 - 16x^3 $ and $b(x,y) =- 14x^3 y^3  + (27x^4 - 9x^2 ) y^2   - (24x^5 - 18x^3 - 12x)y + 14x^6 - 3x^4 -12x^2 -2$ and claim that 

$\bullet$ $a(x,y)<0$ in either $(0,4/5] \times (-1, 0]$ or $((13-\sqrt{106})/21, 4/5] \times [0,1/3]$.

$\bullet$ $b(x,y) <0$ for $(x,y) \in (0,4/5] \times (-1, 1/3]$.\\
 To prove the claim, we proceed as in the proof of the previous lemma, namely we regard both $a$ and $b$ as polynomials of $y$. 

We first observe that the discriminant of $a$ is positive
\begin{equation*}
D(a)= -4 x^4 ( 59x^4 - 66x^2 -9) >0 
\end{equation*}
and hence $a$ has two real roots. We then compute that
\begin{eqnarray*}
a(x,-1)&=&2x^2 (1+x)(7x^2 - 2x-3)<0 \\
a(x,0) &=& 2x^3 ( 7x^2-8)<0 \\
a(x,1/3) &=& \frac{2}{3}x^2 (1+x) ( 21x^2 - 26x +3) 
\end{eqnarray*}
Since $21x^2 - 26x +3 <0$ for $x_0 = (13 - \sqrt{106})/21 <x$, this proves the assertion of the claim for $a$.

\begin{figure}[t]
\begin{subfigure}{0.5\textwidth}
  \centering
  \includegraphics[width=1.0\linewidth]{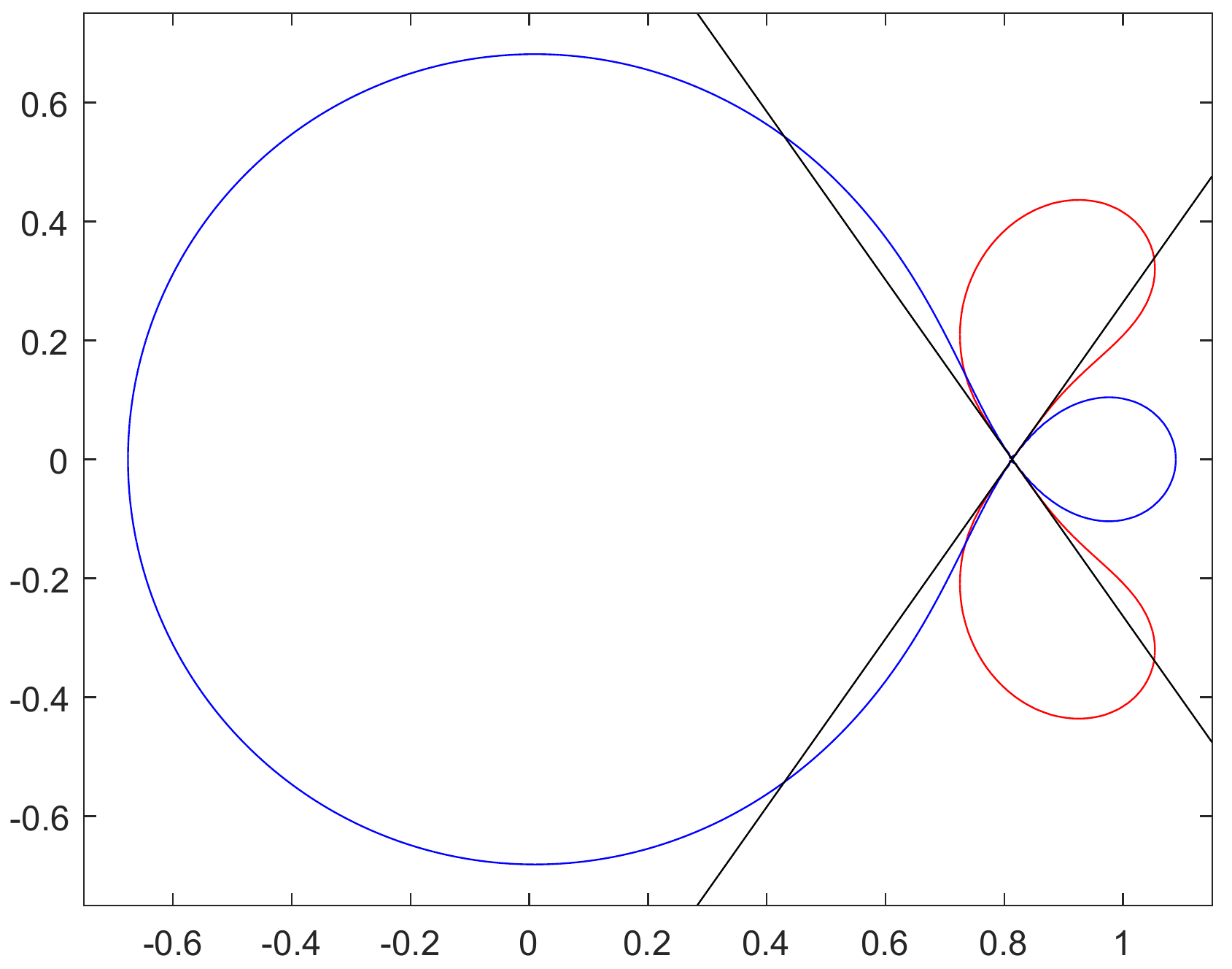}
  \caption{   $\mu = 0.05$    }  
\label{functionff}
\end{subfigure}
\;
~
\begin{subfigure}{0.5\textwidth}
  \centering
  \includegraphics[width=1.0\linewidth]{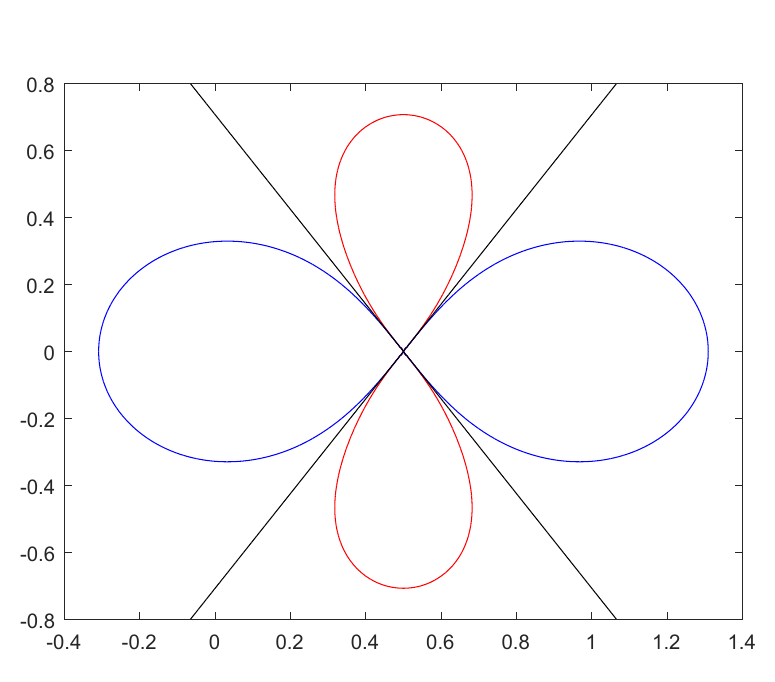}
 \caption{    $\mu=0.5$   }
\label{functionfdf}
\end{subfigure}
  \caption{  The left blue circle is $\mathcal{K}_{c_J}^E$ and the right one is $\mathcal{K}_{c_J}^M$. The red curve represents $C(q) =0$. They intersect at $(l,0)$ and are tangent to the black lines: $q_2 =  \pm \sqrt{2} ( q_1 - l).$}  
 \label{fig1}
\end{figure}

Similarly, we compute the discriminant of $b'$
\begin{equation*}
D(b') = -36x^4 ( 31x^4 - 30x^2 - 65)
\end{equation*}
and since $x \in (0,4/5)$, it is positive. We observe that $b'(x,-1) = -6x(1+x)^2(4x^2 + x - 2) $ is positive for $x<(\sqrt{33}-1)/8$ and negative for  $x>(\sqrt{33}-1)/8$ and $b'(x,1/3) = -\frac{2}{3}x ( 36x^4 - 27x^3 - 20x^3 + 9x-18) >0$. Hence if either $x<(\sqrt{33}-1)/8$ or  $x>(\sqrt{33}-1)/8$, then $b$ attains the maximum at $y=-1$ or $y=1/3$ or at $y=1/3$, respectively.  Now the claim follows from
\begin{eqnarray*}
b(x,-1) &=&  14x^6 +24x^5+24x^4-4x^3-21x^2-12x-2<0\\
b(x,1/3) &=&  14x^6 - 8x^5 + \frac{148}{27}x^3 - 13x^2 +4x -2   <0.
\end{eqnarray*}
This completes the proof of the lemma.
\end{proof}

We are now in a position to prove the theorem. Since the potential function is symmetric with respect to the $q_2$-axis, we may assume that $q_2 \geq 0$. Suppose that there exists another connected component $\Gamma'$ of the curve $C=0$ which intersects $\mathcal{K}_{c_J}^E$. Since $\Gamma'$ is bounded, there exists a critical point of $C$ in the region enclosed by $\Gamma'$. We observe that $q_1^2 + q_2^2 = 1/4$ and $q_2 = -\sqrt{2}(q_1 - 1/2)$ intersect at $(q_1, q_2) = (0,0)$ or $(1/6, \sqrt{2}/3)$ and   $\cos \theta = 1/3$ at $(1/6, \sqrt{2}/3)$. Then by the argument before Lemma \ref{curvaturpositive} and Lemmas     \ref{rderi} and \ref{thetaderi}, the critical point must lie in the region  $\left \{ (q_1, q_2) : q_2 > - \sqrt{2} (q_1- 1/2) \right \}$. Moreover, by Lemma \ref{lemmaposit}, $\Gamma'$ also lies in that region. Together with Corollary \ref{connectednbs} this implies that the curve $C=0$ lies in $\left \{ (q_1, q_2) : q_2 \geq \pm \sqrt{2} (q_1 - 1/2) \right \}.$  By means of Corollary \ref{propfiberwise}, this completes the proof of Theorem \ref{theorem1}.

\;\;\;

\end{document}